\documentclass[12pt]{amsart}
\usepackage{amssymb,latexsym}
\usepackage{enumerate}
\newtheorem{theorem}{Theorem}[section]
\newtheorem{definition}[theorem]{Definition}

\newtheorem{proposition}[theorem]{Proposition}
\newtheorem{corollary}[theorem]{Corollary}
\newtheorem{remark}[theorem]{Remark}
\newcommand{\Ba}[1]{\begin{array}{#1}}
\newcommand{\Ea}{\end{array}}
\newcommand{\Be}{\begin{equation}}
\newcommand{\Ee}{\end{equation}}
\newcommand{\Bea}{\begin{eqnarray}}
\newcommand{\Eea}{\end{eqnarray}}
\newcommand{\Beas}{\begin{eqnarray*}}
\newcommand{\Eeas}{\end{eqnarray*}}
\newcommand{\Benu}{\begin{enumerate}}
\newcommand{\Eenu}{\end{enumerate}}
\newcommand{\Bi}{\begin{itemize}}
\newcommand{\Ei}{\end{itemize}}

\newcommand{\BR}{\begin{Remark} \em}
\newcommand{\ER}{\end{Remark}}
\newcommand{\BE}{\begin{example} \em}
\newcommand{\EE}{\end{example}}
\newcounter{remark}

\begin{document}
\baselineskip=17pt

\title[Carleson measures for the Hilbert-Hardy space]{Some Carleson measures for the Hilbert-Hardy space of tube domains over symmetric cones}
\author{DAVID B\'EKOLL\'E, Beno\^it F. Sehba}
\address{University of Ngaound\'er\'e, Faculty of Science, Department of Mathematics, P. O. Box 454, Ngaound\'er\'e, Cameroon}
\email{dbekolle@univ-ndere.cm}
\address{Department of Mathematics, University of Ghana, Legon LG 62, Accra Ghana}
\email{{bfsehba@ug.edu.gh}}


\maketitle

\begin{abstract}
In this note,  we obtain a full characterization of radial Carleson measures for the Hilbert-Hardy space on tube domains over symmetric cones.  For large derivatives, we also obtain a full characterization of the measures for which  the corresponding embedding operator is continuous. Restricting to the case of light cones of dimension three, we prove that by freezing one or two variables, the problem of embedding derivatives of the Hilbert-Hardy space into Lebesgue spaces reduces to the characterization of Carleson measures for Hilbert-Bergman spaces of the upper-half plane or the product of two upper-half planes.
\end{abstract}



\section{Introduction}
All over the text, $T_\Omega = V+i\Omega$ is the tube domain over the irreducible symmetric cone $\Omega$ (a symmetric Siegel domain of type I). We put $n=dim\hskip 1truemm V$ and we denote by $r$ the rank of the cone $\Omega.$ For more on symmetric cones and on tube domains over these cones, we refer to \cite{FK}. We shall adopt the notations of \cite{FK} and call $\Delta$ the determinant function of the symmetric cone $\Omega.$

A typical example of an irreducible symmetric cone is
the Lorentz cone
$\Lambda_n,\,\,n\geq 3,$ of $\mathbb{R}^n,$ i.e. the set defined by
$$\Lambda_n=\{(y_1,\cdots,y_n)\in \mathbb R^n: y_1+y_2>0\,\,{\rm and}\,\,y_1^2-\cdots-y_n^2>0\},$$
which is a symmetric cone of rank $r=2$ and its determinant function is given by the
Lorentz form
$$\Delta(y)=y_1^2-\cdots-y_n^2.$$
The Hardy space $H^p (T_\Omega)$ is the space consisting of holomorphic functions $F$ on $T_\Omega$ which satisfy the estimate
$$||F||_{H^p (T_\Omega)}:= \left (\sup \limits_{y\in \Omega} \int_V |F(x+iy)|^p dx\right )^{\frac 1p}< \infty.$$
For $\nu$ a real number, and $1\le p<\infty$, we recall that the Bergman space $A_\nu^p(T_\Omega)$ is the subspace of the Lebesgue space $L^p(T_\Omega, dV_\nu)$ consisting of holomorphic functions; here $dV_\nu(x+iy)=\Delta^{\nu-\frac nr}(y)dxdy$. We observe that  $A_\nu^p(T_\Omega)$ is nontrivial only if $\nu>\frac nr-1$ (see \cite{DD} and \cite{BBGNPR}).
\vskip .2cm
We call $\langle \cdot,\cdot\rangle$ the scalar product in $V$ with respect to which $\Omega$ is self-dual. Recall that the Box operator $\Box=\Delta(\frac{1}{i}
\frac{\partial}{\partial x})$ is the differential operator of
degree $r$ in $\mathbb R^n$ defined by the equality: \Be
\Box\,[e^{i\langle x|\xi \rangle}]=\Delta(\xi)e^{i\langle x|\xi \rangle}, \quad
x\in \mathbb R^n,\,\xi\in\Omega. \label{bbox} \Ee
\begin{definition}
Let $k$ be a positive integer. A positive Borel measure $\mu$ on $T_\Omega$ is called a $k$-box Carleson measure for $H^p (T_\Omega), \hskip 1truemm 0<p<\infty$ if there exists a positive constant $C= C(p, k, \mu)$ such that
\begin{equation} \label{carlhardy}
\int_{T_\Omega} |\Box^k F(z)|^p d\mu (z) \leq C||F||_{H^p (T_\Omega)}
\end{equation}
for every $F\in H^p (T_\Omega).$ When $k=0$, $\mu$ is just called a Carleson measure for $H^p (T_\Omega)$.
\end{definition}

{\bf Problem :} Characterize the $k$-box Carleson measures for the Hardy space $H^p (T_\Omega), \hskip 1truemm 0<p<\infty.$
\vskip 2truemm
In the one-dimensional case ($n=1, \hskip 2truemm \Omega = (0, \infty)),$ the solution to this problem was provided by L. Carleson (cf. \cite{G}) for $k=0,$ and for $k\neq 0$, the result is due to D. Luecking (cf. \cite{L}). For tube domains over symmetric cones, the problem is still essentially open. Pretty recently, for values of $p, q$ such that $p< q,$ a characterization of $q$-Carleson measures for $H^p (T_\Omega)$, that is the positive measures $\mu$ such that $H^p (T_\Omega)$ embeds continuously into $L^q(T_\Omega, d\mu),$ was obtained in \cite{BST} in terms of boundedness of a kind of balayage of the measure $\mu$. In this note, we are interested in the above question in the case $p=2$. We do not provide a general characterization but restricting to the two following classes of  measures:\\
- radial measures;\\
- products of a Dirac measure and a measure in the lower dimension,\\
we provide a full characterization. In particular, we prove that when our measure is the product of the Dirac measure and a measure in dimension two, then the problem reduces to a characterization of Carleson measures for Hilbert-Bergman spaces of the product of two upper-half planes. For completeness of our paper, we characterize at the end of our work, Carleson embeddings for (vector) weighted Bergman spaces of the product of upper-half planes. 

\section{Preliminary results}
We refer to \cite{BBGNPR}. The determinant function $\Delta$ of the symmetric cone $\Omega$ has a natural holomorphic extension $\Delta (\frac zi)$ to the tube domain $T_\Omega$ which does not vanish on $T_\Omega.$ For $\alpha\in \mathbb R,$ we shall denote $\Delta^\alpha (\frac zi)$ the holomorphic determination of the power which coincides with $\Delta^\alpha (y)$ when {\bf {$z=iy\in i\Omega.$}}

\begin{proposition}\label{estimate}
The  integral
$$I(y, w):= \int_V \left \vert \Delta^{-\alpha} \left (\frac {x+iy -\bar w}i\right )\right \vert^2 dx \quad (y\in \Omega, \hskip 2truemm w\in T_\Omega)$$
converges if and only if $\alpha > \frac nr - \frac 12.$ In this  case, there exists a positive constant $C_\alpha$ such that
$$I(y, w)=C_\alpha \Delta^{-2\alpha +\frac nr} (y+\Im m\hskip 1truemm w).$$
Furthermore, the weighted Bergman kernel functions $$F(z)=\Delta^{-\alpha} (z-\bar w) \quad (w\in T_\Omega)$$ belong to $H^2 (T_\Omega)$ if and only if $\alpha > \frac nr - \frac 12.$ In this case, $$||F||_{H^2 (T_\Omega)}^2 = C_\alpha \Delta^{-2\alpha + \frac nr} (\Im m \hskip 1truemm w).$$
\end{proposition}

Let us recall the following characterization of $H^2 (T_\Omega)$ {\textbf {(cf. e.g. \cite{FK})}}.

\begin{proposition} \label{Paley-Wiener} (Paley-Wiener)
A holomorphic function $F$ on $T_\Omega$ belongs to the Hardy space $H^2 (T_\Omega)$ if and only if there exists a function $f\in L^2 (\Omega)$ such that
\begin{equation}\label{laplace}
F(z) = \int_\Omega f(t)e^{i\langle t, z\rangle}dt \quad \quad \quad (z\in T_\Omega).
\end{equation}
In this case, $||F||_{H^2 (T_\Omega)}^2 = (2\pi)^n\int_\Omega |f(t)|^2 dt.$
\end{proposition}

\begin{definition}
When  a holomorphic function $F$ on $T_\Omega$ can be expressed as in {\rm (\ref{laplace})}, we say that $F$  is the Laplace transform of $f.$
\end{definition}

Let us recall the following Paley-Wiener result for the Bergman spaces $A_\nu^2(T_\Omega)$ (see \cite{BBGNPR} and \cite{FK}).
\begin{proposition} \label{PWBerg} (Paley-Wiener)
Let $\nu>\frac nr-1$. A holomorphic function $F$ on $T_\Omega$ belongs to the Bergman space $A_\nu^2 (T_\Omega)$ if and only if there exists a function $f\in L^2 (\Omega,\frac{dy}{\Delta^{\nu}(y)})$ such that
\begin{equation}\label{laplaceberg}
F(z) = \int_{\Omega} f(t)e^{i\langle t, z\rangle}dt \quad \quad \quad (z\in T_\Omega).
\end{equation}
In this case, $||F||_{A_\nu^2}^2 = (2\pi)^n \Gamma_\Omega (\nu)\int_\Omega |f(t)|^2 \frac{dt}{\Delta^{\nu}(t)},$
where $\Gamma_\Omega$ denotes the gamma function of the cone $\Omega$.
\end{proposition}

We next focus on the upper half-plane $\Pi^+$ of the complex plane $\mathbb C.$

\begin{definition}
Let $\alpha>-1$. We denote by $A^2_\alpha (\Pi^+)$ the weighted Bergman space consisting of holomorphic functions $G$ on $\Pi^+$ satisfying the estimate
$$||G||_{A^2_\alpha (\Pi^+)}:=\left (\int_{\Pi^+} |G(x+iy)|^2 \hskip 1truemm y^{\alpha}dx dy\right )^{\frac 12} <\infty.$$
\end{definition}

We next recall the analogue of the Paley-Wiener theorem for the Bergman space $A^2_\alpha (\Pi^+)$ (cf. \cite{BBGNPR}).

\begin{proposition}\label{PW} 
Let $G$ be a holomorphic function on $\Pi^+.$ The following assertions are equivalent.
\begin{enumerate}
\item
$G$ belongs to the weighted Bergman space $A^2_\alpha (\Pi^+).$ 
\item 
There exists a function $g:(0, \infty)\rightarrow \mathbb C$ satisfying the estimate $\int_0^\infty |g(t)|^2 \frac {dt}{t^{\alpha+1}} <\infty,$ such that
$$G(z) = \int_0^\infty g(t)e^{itz}dt \quad \quad \quad (z\in \Pi^+).$$
\end{enumerate}
In this case, $||G||_{A^2_\alpha (\Pi^+)}^2 =  2\pi \Gamma (\alpha +1)\int_0^\infty |g(t)|^2 \frac {dt}{t^{\alpha+1}},$ where $\Gamma$ is the usual gamma function.
\end{proposition}

\begin{definition}
Let $\vec{\alpha}=(\alpha_1,\alpha_2)$ where $\alpha_1,\alpha_2>-1$. We denote by $A^2_{\vec{\alpha}}  (\Pi^+\times \Pi^+)$ the weighted  Bergman space consisting of holomorphic functions $G$ on $\Pi^+\times \Pi^+,$ which satisfy the following estimate 
$$
\begin{array}{clcr}
||G||_{A^2_{\vec{\alpha}} (\Pi^+\times \Pi^+)}&=\left (\int_{\Pi^+\times \Pi^+} |G(x_1+iy_1, x_2+iy_2)|^2 y_1^{\alpha_1}y_2^{\alpha_2}dx_1 dx_2 dy_1dy_2\right )^{\frac 12}\\
&<\infty.
\end{array}
$$
\end{definition}
When $\alpha_1=\alpha_2=\alpha$, for simplicity, we use the notation $A^2_\alpha  (\Pi^+\times \Pi^+)$ for $A^2_{(\alpha,\alpha)}  (\Pi^+\times \Pi^+)$.
\vskip .2cm
We recall also the following Paley-Wiener result for Bergman spaces of the tube domain $\Pi^+\times \Pi^+$ in $\mathbb C^2$ over the first octant $(0, \infty)\times (0, \infty)$ (cf. \cite{SW}).
\begin{proposition}\label{PWProd} 
Let $\vec{\alpha}=(\alpha_1,\alpha_2)$, with $\alpha_1,\alpha_2>-1$. Let $G$ be a holomorphic function on $\Pi^+\times \Pi^+.$ Then the following assertions are equivalent.
\begin{enumerate}
\item
$G$ belongs to the weighted Bergman space $A^2_{\vec{\alpha}} (\Pi^+\times \Pi^+).$ 
\item 
There exists a function $g:(0, \infty) 	\times (0, \infty)\rightarrow \mathbb C$ satisfying the estimate $\int_{(0, \infty) 	\times (0, \infty)} |g(t_1,t_2)|^2 \frac {dt_1}{t^{\alpha_1+1}} \frac {dt_2}{t^{\alpha_2+1}}<\infty,$ such that
$$G(z) = \int_{(0, \infty)\times (0, \infty)} g(t)e^{i(z_1 t_1 + z_2 t_2)}dt_1 dt_2 \quad \quad \quad (z\in \Pi^+\times \Pi^+).$$
\end{enumerate}
In this case, $||G||_{A^2_\alpha (\Pi^+\times \Pi^+)} =  c_{\vec{\alpha}}\left (\int_{(0, \infty) 	\times (0, \infty)} |g(t_1,t_2)|^2 \frac {dt_1}{t_1^{\alpha_1+1}} \frac {dt_2}{t_2^{\alpha_2+1}}\right )^{\frac 12}$ with $c_{\vec{\alpha}}=2\pi \sqrt {\Gamma (\alpha_1 +1)\Gamma (\alpha_2 +1)}.$
\end{proposition}

\section{Radial Carleson measures for $H^2(T_\Omega)$}
In view of the second assertion of Proposition \ref{estimate}, testing on the weighted Bergman kernel functions $$\Delta^{-\alpha} (z-\bar w) \quad (w\in T_\Omega),$$ we obtain at once the following necessary condition for $\mu$ be a Carleson measure for $H^2 (T_\Omega):$ there exists a positive constant $C=C(\alpha, \mu)$ such that
\begin{equation} \label{nec}
\int_{T_\Omega} |\Delta^{-\alpha} (z-\bar w)|^2 d\mu (z) \leq C\Delta^{-2\alpha + \frac nr} (\Im m \hskip 1truemm w)
\end{equation}
for every $w\in T_\Omega$, provided $\alpha > \frac nr - \frac 12.$ 
\vskip .2cm
The aim of the present section is to investigate whether this necessary condition is also sufficient as in the one-dimensional case.
We consider particular radial measures $\mu$ on $T_\Omega$ of the form $d\mu (x+iy)=\varphi (y)dxdy,$ where $\varphi$ is a positive measurable function on $\Omega.$ 

Using the Plancherel Theorem, one obtains that the Carleson measure property (\ref{carlhardy}) may be expressed in this particular case as
$$
\int_\Omega \left(\int_\Omega |f(t)|^2 e^{-2\langle t, y\rangle}dt\right)\varphi (y)dy = \int_\Omega \left(\int_\Omega \varphi (y)e^{-2\langle t, y\rangle}dy\right)|f(t)|^2 dt\\
$$
\begin{equation}\label{carlhardypart}
\leq C\int_\Omega |f(t)|^2 dt.
\end{equation}

On the other hand, in view of the first assertion of Proposition \ref{estimate}, the necessary condition inequality (\ref{nec}) may be expressed as
\begin{equation} \label{necpart}
\int_\Omega \Delta^{-2\alpha + \frac nr} (y+t)\varphi (y)dy \leq C_{\alpha, \varphi} \Delta^{-2\alpha + \frac nr} (t)
\end{equation}
for every $t\in \Omega.$

We shall prove the following theorem:

\begin{theorem}\label{H}
Let $\varphi$ be a positive measurable function on the cone $\Omega.$ The following three assertions are equivalent.
\begin{enumerate}
\item
The measure $d\mu(x+iy)=dx\varphi(y)dy$ is a Carleson measure for the Hardy space $H^2 (T_\Omega).$
\item
The function $\varphi$ is integrable on $\Omega.$
\item
For some (all) $\alpha > \frac nr - \frac 12, $ there exists a positive constant $C_{\alpha, \varphi}$ such that (\ref{necpart}) holds for every $t\in \Omega.$
\end{enumerate}
\end{theorem}

\begin{proof}
We shall prove the following implications: $(2)\Rightarrow (1) \Rightarrow (3) \Rightarrow (2).$\\
We first show the implication $(2)\Rightarrow (1).$ In view of (\ref{carlhardypart}), the measure $d\mu (x+iy)=dx\varphi (y)dy,$ where $\varphi$ is a positive measurable function on $\Omega,$ is a Carleson measure for $H^2 (T_\Omega)$ if and only if the function $$t\in \Omega \mapsto \left\{\int_\Omega \varphi (y)e^{-2\langle t, y\rangle}dy\right\}^{\frac 12}$$ is a multiplier of $L^2 (\Omega).$ The latter property is valid if and only if the relevant function is bounded on $\Omega.$ It is then clear that assertion (2) implies assertion (1).\\
\indent
Prior to the statement of the theorem, we proved the implication $(1) \Rightarrow (3)$ which amounts to the fact that (\ref{necpart}) is a necessary condition for (1).\\
\indent
We finally show the implication $(3) \Rightarrow (2).$  We recall the following order relation $\prec$ on $\Omega.$ We write $x\prec y$ if $y-x\in \Omega.$ It is well known that $\Delta(x)\leq \Delta(y)$ whenever $x\prec y.$ Consequently, if $y\prec t,$ then $y+t\prec 2t$ and hence $\Delta(t)\leq \Delta(y+t)\leq \Delta (2t)=C\Delta(t).$\\
The assertion (3) may be written as
$$\sup \limits_{t\in \Omega} \int_\Omega \frac {\Delta^{-2\alpha + \frac nr} (y+t)}{\Delta^{-2\alpha + \frac nr} (t)}\varphi (y)dy \leq C_{\alpha, \varphi}.$$
We obtain:
$$\sup \limits_{t\in \Omega}\int_{y\in \Omega: y\prec t} \varphi (y)dy \leq C_\alpha \sup \limits_{t\in \Omega} \int_\Omega \frac {\Delta^{-2\alpha + \frac nr} (y+t)}{\Delta^{-2\alpha + \frac nr} (t)}\varphi (y)dy\leq C_{\alpha, \varphi}.$$
We call $\mathbf {y_0}$ a base point of $\Omega.$ Then
$$\int_\Omega \varphi (y)dy = \lim \limits_{N\rightarrow \infty} \int_{y\in \Omega: y\prec N\mathbf {y_0}} \varphi (y)dy \leq C_\alpha \sup \limits_N \int_{y\in \Omega: y\prec N\mathbf {y_0}} \varphi (y)dy\leq C_{\alpha, \varphi}.$$ 
Here, the equality follows from the Lebesgue monotone convergence theorem. This concludes the proof of the implication $(3) \Rightarrow (2).$ 
\end{proof}
\section{Embedding derivatives of Hardy spaces into Lebesgue spaces}
In \cite{L}, D. Luecking characterized those measures $\mu$ on the upper half-plane $\Pi^+$ of the complex plane such that differentiation $\frac {d^m}{dz^m}$ of order $m=0, 1,\cdots$  maps $H^p (\Pi^+)$ boundedly into $L^q (\Pi^+, d\mu),$ where $0<p, q<\infty.$\\
\indent
In our setting, this question is also still open. Note that in this case, we replace the differential operator $\frac d{dz}$ by the box operator $\Box$ defined in the introduction.  We shall consider here only the case $p=q=2.$ 
\subsection{Embedding large derivatives of $H^2(T_\Omega)$ into Hilbert-Lebesgue spaces}

We deduce the following corollary from Proposition \ref{PWBerg}.

\begin{corollary}
We suppose that $m$ is an integer such that $2m>\frac nr-1$. Then the differential operator $\Box^{m}$ is a bounded isomorphism from $H^2 (T_\Omega)$ to $A_{2m}^2 (T_\Omega).$
\end{corollary}

\begin{proof}
First by Proposition \ref{Paley-Wiener}, each $F\in H^2 (T_\Omega)$ is the Laplace transform of a measurable function $f:\Omega \rightarrow \mathbb C$
with the equality $||F||_{H^2 (T_\Omega)}^2= (2\pi)^n \int_\Omega |f(\xi)|^2d\xi.$ Applying the operator $\Box$ to $F$  $m$ times, we obtain
$$\Box^{m}F (z)=\int_\Omega \Delta^{m}(\xi)e^{i\langle z,\xi\rangle}f(\xi)d\xi, \quad z\in T_\Omega.$$
Put $g(\xi):=\Delta^{m}(\xi)f(\xi)$. In view of Proposition \ref{PWBerg}, we have 
$$
\begin{array}{clcr}
||\Box^{m}F||_{A_{2m}^2 (T_\Omega)}^2&=(2\pi)^n \Gamma_\Omega (2m)\int_\Omega |g(\xi)|^2\frac {d\xi}{\Delta^{2m}(\xi)}\\
&=(2\pi)^n \Gamma_\Omega (2m)\int_\Omega |f(\xi)|^2d\xi\\
&=\Gamma_\Omega (2m)||F||_{H^2 (T_\Omega)}^2.
\end{array} 
$$
Conversely, if $G\in A_{2m}^2(T_\Omega)$, then by Proposition \ref{PWBerg},
$$G (z)=\int_\Omega g(\xi)e^{i\langle z,\xi\rangle}d\xi, \quad z\in T_\Omega$$
for some $g\in L^2 \left (\Omega, \frac{d\xi}{\Delta^{2m}(\xi)}\right )$, with $||G||_{A_{2m}^2 (T_\Omega)}^2=(2\pi)^n \Gamma_\Omega (2m)\int_{\Omega}|g(\xi)|^2\frac{d\xi}{\Delta^{2m}(\xi)}$. 
\vskip .1cm
Put $f(\xi):=\Delta^{-m}(\xi)g(\xi)$. Then $f\in L^2(\Omega, d\xi)$ and if we define $F$ by 
 $$F (z)=\int_\Omega f(\xi)e^{i\langle z,\xi\rangle}d\xi, \quad z\in T_\Omega,$$
 then $F$ is well-defined and one easily checks that $\Box^mF=C_mG$, and $||F||_{H^2 (T_\Omega)}^2= (2\pi)^n \int_\Omega |f(\xi)|^2d\xi=(2\pi)^n\|g\|_{L^2 \left (\Omega, \frac{d\xi}{\Delta^{2m}(\xi)}\right )}^2.$
The proof is complete.
\end{proof}

We then obtain the following characterization of $m$-box Carleson measures for $H^2(T_\Omega)$ for large integer $m$.

\begin{theorem}
We suppose that $m$ is an integer such that $2m>\frac nr-1$. Let $\mu$ be a positive Borel measure on $T_\Omega.$ Then the following two assertions are equivalent.
\begin{enumerate}
\item
There exists a positive constant $A=A(m)$ such that
$$\int_{T_\Omega} \left \vert \Box^{m}F\right \vert^2d\mu \leq A||F||_{H^2 (T_\Omega)}^2$$
for each $F\in H^2 (T_\Omega);$
\item
$\mu$ is a Carleson measure for the weighted Bergman spaces $A_{2m}^2 (T_\Omega).$
\end{enumerate}
\end{theorem}

\begin{proof}
According to the previous corollary, assertion (1) is equivalent to the following assertion: there exists a positive constant $C$ such that
$$\int_{T_\Omega} |G|^2d\mu \leq C||G||_{A_{2m}^2 (T_\Omega)}^2$$
for each $G\in A_{2m}^2 (T_\Omega).$ The latter assertion is clearly assertion (2).

\end{proof}

We recall that the Carleson measures for Bergman spaces on tube domains over symmetric cones were characterized in \cite{NS} in terms of a geometrical condition on Bergman balls..
\vskip .3cm
When the integer $m$ is such that $0\le m\le \frac 12(\frac nr-1)$, the above techniques do not provide any answer. Nevertheless, in the following, we provide a full characterization when considering only some restricted measures in the setting of the tube domain over the Lorentz cone of dimension three.

\subsection{Two examples of $m$-box Carleson measures for the Hardy space $H^2$ on the tube domain over the Lorentz cone $\Lambda_3$ of $\mathbb R^3$}
{\textbf {1. First class of examples: measures of the form $\mu(z_1)\delta_O (z_2, z_3).$}}
We denote by $\delta_O (z_2, z_3)$ the Dirac measure at the origin in $\mathbb C\times \mathbb C.$ We shall characterize the positive measures $\mu(z_1)$ on the upper half-plane $\Pi^+$ for which there exists a positive constant  $C$ such that for each $F\in H^2 (T_{\Lambda_3}),$ the following estimate holds.
$$
\begin{array}{clcr}
\int_{T_{\Lambda_3}} |\Box^mF(z_1,z_2,z_3)|^2 d\mu (z_1)d\delta_O (z_2, z_3) &= \int_{\Pi^+} |(\Box^mF)(z_1, 0, 0)|^2 d\mu(z_1)\\
& \leq C||F||^2_{H^2 (T_{\Lambda_3})}.
\end{array}$$

\begin{definition}
\begin{itemize}
\item[(a)]
Let $F$ be a complex-valued function defined on $T_{\Lambda_3}.$ We call restriction of F to $\Pi^+,$ the function $RF: \Pi^+ \rightarrow \mathbb C$ defined by
$$(RF)(z_1)=F(z_1,0,0).$$
\item[(b)]
Let $G$ be a complex-valued  function  defined on $\Pi^+.$ We say that the function $F:T_{\Lambda_3}\rightarrow \mathbb C$ is an extension of $G$ if $RF=G.$
\end{itemize}
\end{definition}

We shall use the following result.

\begin{proposition}\label{restrict}
\begin{itemize}
\item[(1)]
The following estimate holds:  
$$||R\Box^mF||_{A^2_{4m+1} (\Pi^+)}\leq \frac {\Gamma (4m+2)}{2^{4m+4}\pi (2m+1)} ||F||^2_{H^2 (T_{\Lambda_3})}$$
for every $F\in  H^2 (T_{\Lambda_3})$.
\item[(2)]
Conversely, for every function $G\in A^2_{4m+1} (\Pi^+)$, there exists a function $F\in H^2 (T_{\Lambda_3})$ such that $\Box^mF$ is an extension of $G$. Moreover,
$$||F||_{H^2 (T_{\Lambda_3})}=\frac{m+1}{\pi \sqrt{2\Gamma(4m+2)}}||G||_{A^2_{4m+1} (\Pi^+)}.$$ 
\end{itemize}
\end{proposition}

\begin{proof}

(1): We have $\langle z, t\rangle:=z_1t_1+z_2t_2+z_3t_3.$ Let $F\in  H^2 (T_{\Lambda_3}).$ In view of Proposition \ref{Paley-Wiener}, there exists a measurable function $f:\Lambda_3 \rightarrow \mathbb C$ such that
$$F(z) = \int_{\Lambda_3} f(t)e^{i\langle z, t\rangle}dt \quad \quad (z\in T_{\Lambda_3}),$$
with $|F||_{H^2 \left (T_{\Lambda_3}\right )}^2=8\pi^3 \int_{\Lambda_3} |f(t)|^2dt.$ Then
$$
\begin{array}{clcr}
(R\Box^mF)(x_1+iy_1)&=\int_{\Lambda_3} f(t_1, t_2, t_3)e^{it_1(x_1+iy_1)}\Delta^m(t)dt_1dt_2dt_3\\
& =\int_0^{\infty} \left (\int_{t_2^2+t_3^2 < t_1^2} f(t_1, t_2, t_3)\Delta^m(t)dt_2dt_3\right )e^{it_1(x_1+iy_1)}dt_1.
\end{array}
$$
An application of the Plancherel formula implies
$$\int_{-\infty}^{\infty} \left |(R\Box^mF)(x_1+iy_1)\right |^2dx_1 = 2\pi\int_0^{\infty} \left |\int_{t_2^2+t_3^2 < t_1^2} f(t_1, t_2, t_3)\Delta^m(t)dt_2dt_3\right |^2e^{-2t_1y_1}dt_1.$$ 
Next, when applying the Fubini Theorem, we obtain:

\Beas
||R\Box^mF||_{A^2_{4m+1} (\Pi^+)}^2 &=& \int_0^{\infty} \left (\int_{-\infty}^{\infty} \left |(R\Box^mF)(x_1+iy_1)\right |^2dx_1\right )y_1^{4m+1}dy_1 \\
&=& 2\pi\int_0^{\infty} \left |\int_{t_2^2+t_3^2 < t_1^2} f(t_1, t_2, t_3)\Delta^m(t)dt_2dt_3\right |^2\\ & &\left (\int_0^{+\infty} e^{-2t_1y_1}y_1^{4m+1}dy_1\right )dt_1\\
&=& \frac{\pi \Gamma (4m+2)}{2^{4m+1}}\int_0^{\infty} \left |\int_{t_2^2+t_3^2 < t_1^2} f(t_1, t_2, t_3)\Delta^m(t)dt_2dt_3\right |^2\frac {dt_1}{t_1^{4m+2}}.
\Eeas

We finally applying the Schwarz inequality to the integral with respect to $dt_2dt_3$, we get

\Beas 
&& \left \vert \int_{t_2^2+t_3^2 < t_1^2} f(t_1, t_2, t_3)\Delta^m(t)dt_2dt_3)\right \vert^2\\
&\leq& \left (\int_{t_2^2+t_3^2 < t_1^2} |f(t_1, t_2, t_3)|^2dt_2dt_3\right )\left (\int_{t_2^2+t_3^2 < t_1^2}\Delta^{2m}(t)dt_2dt_3\right )\\
&=& \frac \pi{2m+1} t_1^{4m+2} \int_{t_2^2+t_3^2 < t_1^2} |f(t_1, t_2, t_3)|^2dt_2dt_3.
\Eeas

We conclude that
\Beas
||R\Box^mF||_{A^2_{4m+1} (\Pi^+)}^2 &\leq& \frac {\pi^2 \Gamma (4m+2)}{2^{4m+1}(2m+1)} \int_0^{\infty} (\int_{t_2^2+t_3^2 < t_1^2} |f(t_1, t_2, t_3)|^2dt_2dt_3)dt_1\\
&=& \frac {\pi^2 \Gamma (4m+2)}{2^{4m+1}(2m+1)} \int_\Omega |f(t)|^2dt\\ &=& \frac { \Gamma (4m+2)}{2^{4m+4}\pi (2m+1)}|F||^2_{H^2 (T_{\Lambda_3})}.
\Eeas

(2): Conversely, let $G\in A^2_{4m+1} (\Pi^+).$ In view of Proposition \ref{PW}, there exists a measurable function $g:(0, \infty) \rightarrow \mathbb C$ such that
$$G(z_1) =\int_0^{\infty} g(t_1)e^{it_1z_1}dt_1 \quad \quad (z_1\in \Pi^+),$$
with $||G||_{A^2_{4m+1} (\Pi^+)}^2=2\pi \Gamma (4m+2)\int_0^{\infty} |g(t_1)|^2\frac {dt_1}{t_1^{4m+2}}.$
Let 
$$F(z):=\frac {m+1}{\pi} \int_\Omega \frac {g(t_1)}{t_1^{2m+2}} e^{i(z_1t_1+z_2t_2+z_3t_3)}dt_1dt_2dt_3.$$
We have 
\Beas
R\Box^mF(z) &=& \frac {m+1}{\pi}\int_\Omega \frac {g(t_1)}{t_1^{2m+2}} e^{iz_1t_1}\Delta^m(t)dt_1dt_2dt_3\\ &=& \frac {m+1}{\pi}\int_0^\infty \frac {g(t_1)}{t_1^{2m+2}} e^{iz_1t_1}\left(\int_{t_2^2+t_3^2 < t_1^2}\Delta^m(t)dt_2dt_3\right)dt_1\\ &=& \int_0^{\infty} g(t_1)e^{it_1z_1}dt_1.
\Eeas
That is,  $\Box^mF$ is an extension of $G$ to $T_{\Lambda_3}.$ Let us prove that $F\in H^2 (T_{\Lambda_3}).$ In view of Proposition \ref{Paley-Wiener}, it suffices to prove that
$$||F||_{H^2 \left (\Lambda_3 \right )}^2=\frac {(m+1)^2}{\pi^2}\int_\Omega \frac {|g(t_1)|^2}{t_1^{4m+4}} dt_1dt_2dt_3 <+\infty.$$
Proceeding as above, we obtain that 
\Beas
\int_\Omega \frac {|g(t_1)|^2}{t_1^{4m+4}} dt_1dt_2dt_3 &=& \int_0^\infty \frac {|g(t_1)|^2}{t_1^{4m+4}}\left(\int_{t_2^2+t_3^2<t_1^2}dt_2dt_3\right)dt_1\\ &=& \pi\int_0^\infty \frac {|g(t_1)|^2}{t_1^{4m+2}}dt_1\\ &=& \frac 1{2\Gamma(4m+2)}||G||^2_{A^2_{4m+2} (\Pi^+)}.
\Eeas

Hence
$$||F||_{H^2 \left (\Lambda_3 \right )}^2=\frac {(m+1)^2}{2\pi^2\Gamma(4m+2)}||G||^2_{A^2_{4m+2} (\Pi^+)}<\infty.$$
\end{proof}

Our main result in this subsection is the following.

\begin{theorem}\label{delta}
Let $\mu$ be a positive measure on $\Pi^+$, and $m\ge 0$ an integer. Then the following two properties are equivalent.
\begin{itemize}
\item[(1)]
$\mu (z_1)\delta_O (z_2, z_3)$ is a $m$-box Carleson measure for the Hardy space $H^2 (T_{\Lambda_3});$
\item[(2)]
$\mu$ is the Carleson measure for the weighted Bergman space $A^2_{4m+1} (\Pi^+).$
\end{itemize}
\end{theorem}

\begin{proof}
$(2)\Rightarrow (1)$ Let $\mu$ be a  Carleson measure for the weighted Bergman space  $A^2_{4m+1} (\Pi^+)$ with Carleson constant $C.$  Then
for every $F\in H^3 (T_{\Lambda_3}),$ we have
\Beas
\int_{T_{\Lambda_3}} |\Box^m F(z)|^2d\mu (z_1)d\delta_O (z_2, z_3) &=& \int_{\Pi^+} |(\Box^m F)(z_1,0,0)|^2 d\mu(z_1)\\
&=& \int_{\Pi^+} |(R\Box^m F)(z_1)|^2 d\mu(z_1)\\ &\leq& C||R\Box^m F||_{A^2_{4m+2} (\Pi^+)}^2\\
&\leq& \frac { C\Gamma (4m+2)}{2^{4m+4}\pi (2m+1)}||F||_{H^2 (T_{\Lambda_3})}^2,
\Eeas

where the latter inequality follows from assertion (1) of Proposition \ref{restrict}. \\ 
\indent
$(1)\Rightarrow (2)$ We suppose that $\mu (z_1)\delta_O (z_2, z_3)$ is a $m$-box Carleson measure for the Hardy space $H^2 (T_{\Lambda_3})$ with Carleson constant $C.$ It follows from assertion (2) of Proposition \ref{restrict} that for every function $G\in A^2_{4m+1} (\Pi^+)$, there is an  $F\in H^2 (T_{\Lambda_3})$  such that $\Box^mF$ is an extension of $G$ to $T_{\Lambda_3}$ and  $||F||_{H^2 (T_{\Lambda_3})}=2\frac {m+1}{\pi \sqrt{2\Gamma (4m+2)}} ||G||_{A^2_{4m+1} (\Pi^+)}.$ Thus
$$ 
\begin{array}{clcr}
\int_{\Pi^+} |G(z_1)|^2d\mu(z_1)&=\int_{T_{\Lambda_3}} |\Box^mF(z_1, z_2, z_3)|^2d\mu (z_1)d\delta_O (z_2, z_3)\\
&\leq C||F||_{H^2 \left (T_{\Lambda_3}\right )}^2=\frac {C(m+1)^2}{2\pi^2\Gamma (4m+2)}||G||^2_{A^2_{4m+1} (\Pi^+)}.
\end{array}
$$
The proof is complete.
\end{proof}

We recall that a characterization of Carleson measures for standard Bergman spaces $A^2_\alpha$ of the unit disc was provided by D. Stegenga \cite{St} in terms of a geometrical condition on Carleson sectors. For the unweighted case, refer to \cite{H}. A characterization of Carleson measures for standard Bergman spaces  $A^2_\alpha (\Pi^+)$ in terms of a geometrical condition on Carleson rectangles can be found in \cite{CarBen} (see also Section 5 below).
\vskip 2truemm
\noindent
{\textbf {2. Second class of examples:  measures of the form $\mu(z_1, z_2)\delta_0 (z_3).$}}
We denote by $\delta_0 (z_3)$ the Dirac measure at the origin in $\mathbb C.$ We characterize the positive measures $\mu(z_1, z_2)$ on the domain $D:=\{(x_1+iy_1, x_2+iy_2)\in \mathbb C^2: y^2_1 >y_2^2, \hskip 1truemm y_1>0 \}$ of $\mathbb C^2$ for which there exists a positive constant  $C$ such that the following estimate holds:
$$
\begin{array}{clcr}
\int_{T_{\Lambda_3}} |\Box^mF(z_1,z_2,z_3)|^2d\mu (z_1, z_2)\delta_0 (z_3)\\
= \int_{D} |(\Box^mF)(z_1, z_2, 0)|^2d\mu(z_1, z_2)
 \leq C||F||^2_{H^2 (T_{\Lambda_3})}
\end{array}
$$
for every $F\in H^2 (T_{\Lambda_3}).$ It is easily checked that the  Lorentz cone ${\Lambda_3}$ is linearly equivalent to the spherical cone $\Sigma$ in $\mathbb R^3$ defined by
$$\Sigma :=\{(y_1, y_2, y_3)\in \mathbb R^3: y_1y_2 -y_3^2 >0, \hskip 1truemm y_1 >0\}$$
Our problem takes the following form. Characterize the positive measures  $\mu (z_1, z_2)$ on the  product $\Pi^+\times \Pi^+$ of two upper half-planes (the tube domain over the first octant)
for which there exists a positive constant   $C$ such that the following estimate holds:
$$
\begin{array}{clcr}
\int_{T_\Sigma} |\Box^mF(z_1,z_2,z_3)|^2d\mu (z_1, z_2)\delta_0 (z_3)\\
= \int_{\Pi^+\times \Pi^+} |(\Box^mF)(z_1, z_2, 0)|^2d\mu(z_1, z_2)
\leq C||F||^2_{H^2 (T_S)}
\end{array}
$$
for every  $F\in H^2 (T_\Sigma).$ We note that in this case, the determinant function is now $\Delta(t)=t_1t_2-t_3^2$.

\begin{definition}
\begin{itemize}
\item[(1)]
Let $F$ be a complex-valued function  defined on   $T_\Sigma.$ We call restriction of F to $\Pi^+\times \Pi^+$ the function $RF: \Pi^+ \times \Pi^+\rightarrow \mathbb C$ defined by 
$$(RF)(z_1, \hskip 1truemm z_2)=F(z_1,z_2,0).$$
\item[(2)]
Let $G$ be a complex-valued function  defined on $\Pi^+\times \Pi^+.$ We say that a function $F:T_\Sigma \rightarrow \mathbb C$ is an extension of $G$ if $RF=G.$
\end{itemize}
\end{definition}

Our result is the following. 

\begin{theorem}
Let $\mu$ be a positive measure on $\Pi^+\times \Pi^+$, and let $m\ge 0$ be an integer. Then the following two assertions are equivalent.
\begin{itemize}
\item[(1)]
$\mu (z_1, z_2)\delta_0 (z_3)$ is a $m$-box Carleson measure for the Hardy space $H^2 (T_\Sigma)$
\item[(2)]
$\mu$ is a Carleson measure for the weighted Bergman space\\ $A^2_{2m-\frac 12} (\Pi^+\times \Pi^+).$
\end{itemize}
\end{theorem}
The proof is an easy adaptation of the proof of Theorem \ref{delta} with the help of the following result.

\begin{proposition}\label{restrictprod}
\begin{itemize}
\item[(1)]
The following estimate holds:  
$$||R\Box^mF||_{A^2_{2m-\frac 12} (\Pi^+\times \Pi^+)}^2\leq \frac {4\pi^{\frac 52}(2m)!\Gamma \left (2m+\frac 12\right )}{2m+\frac 12}||F||^2_{H^2 (T_\Sigma)}$$
for every $F\in  H^2 (T_\Sigma)$.
\item[(2)]
Conversely, for every function $G\in A^2_{2m-\frac 12} (\Pi^+\times \Pi^+)$, there exists a function $F\in H^2 (T_\Sigma)$ such that $\Box^mF$ is an extension of $G$. Moreover,
$$||F||_{H^2 (T_\Sigma)}^2=\frac {\left (\Gamma \left (m+\frac 32\right )\right )^2}{4\pi^2\left (\Gamma \left (2m+\frac 12\right )m!\right )^2}||G||_{A^2_{2m-\frac 12} (\Pi^+\times \Pi^+)}^2.$$ 
\end{itemize}
\end{proposition}

\begin{proof}[Proof of Proposition \ref {restrictprod}]
(1): We recall that for $z=(z_1,z_2,z_3)\in T_\Sigma$ and $t=(t_1,t_2,t_3)\in \Sigma$, we have $\langle z, t\rangle:=\frac 12 (z_1t_1+z_2t_2)+z_3t_3.$ Let $F\in  H^2 (T_\Sigma).$ We recall with Proposition \ref{Paley-Wiener} that there exists a measurable function $f:\Sigma \rightarrow \mathbb C$ such that
$$F(z) = \int_{\Sigma} f(t)e^{i\langle z, t\rangle}dt \quad \quad (z\in T_{\Sigma}),$$
with $||F||_{H^2 (T_\Sigma)}^2=8\pi^3\int_{\Sigma} |f(t)|^2dt.$ Then 
\Beas
(R\Box^mF)(z_1,z_2) =\int_{\Sigma} f(t_1, t_2, t_3)e^{\frac i2 (t_1z_1+t_2z_2)}\Delta^m(t)dt_1dt_2dt_3\\
= \int_{(0,\infty)\times (0, \infty)} \left (\int_{t_3^2 < t_1t_2} f(t_1, t_2, t_3)\Delta^m(t)dt_3\right )e^{\frac i2 (t_1z_1+t_2z_2)}dt_1dt_2.
\Eeas
Using Plancherel formula, we obtain
\Beas
 \int_{\mathbb{R}^2} \left \vert (R\Box^mF)(x_1+iy_1,x_2+iy_2)\right \vert^2dx_1dx_2\\ 
= 4\pi^2\int_{(0, \infty) 	\times (0, \infty)} \left |\int_{t_3^2 < t_1t_2} f(t_1, t_2, t_3)\Delta^m(t)dt_3\right |^2e^{-(t_1y_1+t_2y_2)}dt_1dt_2.
\Eeas
Next, applying the Fubini Theorem, we obtain:
$$
\begin{array}{clcr}
||R\Box^mF||_{A^2_{2m-\frac 12}}^2\\
= \int_{(0,\infty)\times (0, \infty)} \left(\int_{\mathbb{R}^2} \left |(R\Box^mF)(x_1+iy_1,x_2+iy_2)\right |^2dx_1dx_2\right)\times\\ 
y_1^{2m-\frac 12}y_2^{2m-\frac 12}dy_1dy_2 \\
= 4\pi^2\int_{(0,\infty)\times (0, \infty)} \left |\int_{t_3^2 < t_1t_2} f(t_1, t_2, t_3)\Delta^m(t)dt_3\right |^2 \times\\
  \left(\int_{(0,\infty)\times (0, \infty)}e^{-(t_1y_1+t_2y_2)} y_1^{2m-\frac 12}y_2^{2m-\frac 12}dy_1dy_2\right)dt_1dt_2\\
= \left (2\pi\Gamma (2m+\frac 12)  \right )^2\int_ {(0,\infty)\times (0, \infty)}\left |\int_{t_3^2 < t_1t_2} f(t_1, t_2, t_3)\Delta^m(t)dt_3\right |^2\frac {dt_1}{t_1^{2m+\frac 12}}\frac {dt_2}{t_2^{2m+\frac 12}}.
\end{array}
$$

Apply the Schwarz inequality to the integral with respect to $dt_3;$ we obtain
$$
\begin{array}{clcr} \label{beta}
\left |\int_{t_3^2 < t_1t_2} f(t_1, t_2, t_3)\Delta^m(t)dt_3)\right |^2\\
\leq \left(\int_{t_3^2 < t_1t_2} |f(t_1, t_2, t_3)|^2dt_3\right)\left(\int_{t_3^2 < t_1t_2}\Delta^{2m}(t)dt_3\right)\\
= \beta \left (2m+1, \frac 12\right )  t_1^{2m+\frac 12} t_2^{2m+\frac 12}\int_{t_3^2 < t_1t_2} |f(t_1, t_2, t_3)|^2dt_3.
\end{array}
$$
In the last inequality, we have used that for $k>0$ an integer, and $a>0$,
\Be\label{eq:identity}\int_0^{\sqrt a}(a-x^2)^kdx=c_k a^{k+\frac 12}.\Ee 
More precisely, we have $c_k=\frac{1}{2}\beta \left (k+1, \frac 12\right ),$ where $\beta$ denotes the usual beta function.
We conclude that
$$
\begin{array}{clcr}
||R\Box^mF||_{A^2_{2m-\frac 12} (\Pi^+\times \Pi^+)}^2\\
\leq \frac {4\pi^{\frac 52}(2m)!\Gamma \left (2m+\frac 12\right )}{2m+\frac 12} \int_{(0, \infty) 	\times (0, \infty)} \left(\int_{t_3^2 < t_1t_2} |f(t_1, t_2, t_3)|^2dt_3\right)dt_2dt_1\\
=  \frac {4\pi^{\frac 52}(2m)!\Gamma \left (2m+\frac 12\right )}{2m+\frac 12} \int_\Sigma |f(t)|^2dt=\frac {(2m)!\Gamma (2m+\frac 12}{2\sqrt \pi \left (2m+\frac 12\right )}  |F||^2_{H^2 (T_{\Sigma})}.
\end{array}
$$

(2): Now $G\in A^2_{2m-\frac 12} (\Pi^+\times \Pi^+).$ We know from Proposition \ref{PWProd} that there exists a measurable function $g:(0, \infty) 	\times (0, \infty) \rightarrow \mathbb C$ such  that
$$G(z) =\int_{(0,\infty)\times (0, \infty)} g(t_1,t_2)e^{i(t_1z_1+t_2z_2)}dt_1dt_2 \quad \quad (z=(z_1,z_2)\in \Pi^+\times \Pi^+),$$
with 
$$
\begin{array}{clcr}
||G||_{A^2_{2m-\frac 12} (\Pi^+\times \Pi^+)}^2\\
=4\pi^2 \left (\Gamma \left (2m+\frac 12\right )  \right )^2\int_{(0,\infty)\times (0, \infty)} |g(t_1,t_2)|^2\frac {dt_1}{t_1^{2m+\frac 12}}\frac {dt_1}{t_1^{2m+\frac 12}}.
\end{array}$$
Let 
$$F(z):=\frac 1{2c_m}  \int_\Sigma \frac {g(t_1,t_2)}{t_1^{m+\frac 12}t_2^{m+\frac 12}} e^{i\langle z,t\rangle}dt_1dt_2dt_3,$$
where $c_m$ is the constant of (\ref{eq:identity}).
Using again (\ref{eq:identity}), one easily obtains that $$R\Box^mF(z_1,z_2)=G(z_1,z_2)$$. That is, the function $\Box^mF$ is an extension of $G$ to $T_{\Sigma}.$
\vskip .1cm
Now to conclude, in view of Proposition \ref{Paley-Wiener}, it is enough to prove that
$$\int_\Sigma \frac {|g(t_1,t_2)|^2}{t_1^{2m+1}t_2^{2m+1}} dt_1dt_2dt_3 <\infty.$$
This is obvious since the left hand side of the latter is equal to 
$$
\begin{array}{clcr}
 2\int_{(0,\infty)\times (0, \infty)} |g(t_1,t_2)|^2 \frac {dt_1}{t_1^{2m+\frac 12}}\frac {dt_2}{t_2^{2m+\frac 12}}\\
= \frac {1}{2\pi^2\left (\Gamma \left (2m+\frac 12\right )\right )^2}||G||^2_{A^2_{2m-\frac 12} (\Pi^+\times \Pi^+)}<\infty
\end{array}
$$
{\bf {by Proposition \ref{PWProd}}}. The proof is complete.
\end{proof}
For completeness, we give a characterization of Carleson measures for weighted Bergman space of $\Pi^+\times \Pi^+$ in the next section. 
For a characterization of Carleson measures on the bi-disc, the reader can consult \cite{H} for  unweighted Bergman spaces and \cite{J} for standard weighted Bergman spaces.

\section{Carleson measures for Bergman spaces of the tube over the first octant}
As in the one-parameter case \cite{CarBen}, we apply techniques of real harmonic analysis as developed for example in \cite{CD,GR,SE}. Let denote by $\mathcal{I}$ the set of all intervals of $\mathbb{R}$, and by $\mathcal{R}$ the set of all rectangles in $\mathbb{R}^2$. For $R\in \mathcal{R}$, $R=I_1\times I_2$ where the $I_j\in \mathcal{I}$, $j=1,2$.
\vskip .2cm
For $I$ an interval, we recall that the Carleson square associated to $I$ and denoted $Q_I$ is the set defined by $$Q_I:=\{z=x+iy\in \Pi^+:x\in I,\,\,\,\textrm{and}\,\,\,0<y<|I|\}.$$
The upper half of the Carleson square $Q_I$ is the set 
$$T_I:=\{z=x+iy\in \Pi^+:x\in I,\,\,\,\textrm{and}\,\,\,\frac{|I|}{2}<y<|I|\}.$$
Given $\alpha>-1$, we define the maximal function $M_\alpha$ on $\Pi^+$ by
$$M_\alpha f(z):=\sup_{I\in \mathcal{I}}\frac{\chi_{Q_I}(z)}{V_\alpha(Q_I)}\int_{Q_I}|f(z)|dV_\alpha(z)$$
where for simplicity, we have used the notation $dV_\alpha(x+iy)=y^\alpha dxdy$.
\vskip .2cm
It is not hard to prove that for any $p\in (1,\infty]$ (see for example \cite{CarBen}),
$$\|M_\alpha f\|_{p,\alpha}\le c_{p,\alpha}\|f\|_{p,\alpha}.$$
For $R\in \mathcal{R}$, we define the Carleson box (rectangle) $Q_R$ by $$Q_R=Q_{I_1}\times Q_{I_2},\,\,\,\textrm{whenever}\,\,\,R=I_1\times I_2.$$

On the product $\Pi^+\times \Pi^+$ of two upper-planes, we define the product measure $\mathcal V_{\vec{\alpha}}, \hskip 2truemm \vec{\alpha} = (\alpha_1, \alpha_2)$ by $d\mathcal V_{\vec{\alpha}} (z_1, z_2) =dV_{\alpha_1}(z_1)dV_{\alpha_2}(z_2).$ The weighted Lebesgue space $L^p_{\vec{\alpha}}(\Pi^+\times \Pi^+)$ is the space of measurable functions $f$ on $\Pi^+\times \Pi^+$ such that
$$||f||_{p, \vec{\alpha}}:=\left (\int_{\Pi^+\times \Pi^+} |f|^p d\mathcal V_{\vec{\alpha}}\right )^{\frac 1p}<\infty.$$
The weighted Bergman space $A^p_{\vec{\alpha}}(\Pi^+\times \Pi^+)$ is the subspace of $L^p_{\vec{\alpha}}(\Pi^+\times \Pi^+)$ consisting of holomorphic functions.

\begin{definition}
Given $\vec{\alpha}=(\alpha_1,\alpha_2)$, $\alpha_1,\alpha_2>-1$, and $1<p\leq q<\infty$, we say a positive measure $\mu$ defined on $\Pi^+\times \Pi^+$ is a $(\frac qp,\vec{\alpha})$-Carleson measure, if there is a constant $C>0$ such that for any $R\in \mathcal{R}$,
\Be\label{eq:carlmeaprod}
\mu(Q_R)\le C\left(\mathcal V_{\vec{\alpha}}(Q_R)\right)^{\frac qp}.
\Ee
Here $\mathcal V_{\vec{\alpha}}(Q_R)=V_{\alpha_1}(Q_{I_1})V_{\alpha_2}(Q_{I_2})$ whenever $R=I_1\times I_2$.
\end{definition}
We have the following characterization of Carleson measures for the weighted Bergman spaces in the product of two upper-half planes.
\begin{theorem}\label{thm:carlbergprod}
Let $\vec{\alpha}=(\alpha_1,\alpha_2)$, with $\alpha_1,\alpha_2>-1$, and let $1<p\leq q<\infty$. Assume $\mu$ is a positive measure on $\Pi^+\times  \Pi^+$. Then the following assertions are equivalent.
\begin{itemize}
\item[(a)] $\mu$ is $(\frac qp,\vec{\alpha})$-Carleson measure.
\item[(b)] There exists a constant $C>0$ such that for any $f\in A_{\vec{\alpha}}^p(\Pi^+	\times  \Pi^+)$,
\Be\label{eq:carlembedprod}
\int_{\Pi^+	\times  \Pi^+}|f(z)|^qd\mu(z)\le C\left(\int_{\Pi^+	\times  \Pi^+}|f(z)|^pd\mathcal V_{\vec{\alpha}}(z)\right)^{\frac qp}
\Ee
where $d\mathcal V_{\vec{\alpha}}(z_1,z_2)=dV_{\alpha_1}(z_1)dV_{\alpha_2}(z_2)$.
\end{itemize}
\end{theorem}
We define the strong maximal function on $\Pi^+	\times  \Pi^+$ to be the operator 
$$\mathcal{M}_{\vec{\alpha}}f(z):=\sup_{R\in \mathcal{R}}\frac{\chi_{Q_R}(z)}{\mathcal V_{\vec{\alpha}}(Q_R)}\int_{Q_R}|f(w)|d\mathcal V_{\vec{\alpha}}(w).$$
We observe that $$\mathcal{M}_{\vec{\alpha}}f\le M_{\alpha_1}\circ M_{\alpha_2}f$$
where $M_{\alpha_j}$ is the one-parameter maximal function on $\Pi^+$. It follows from the boundedness of $M_{\alpha_j}$ that for any $1<p\le \infty$,
\Be\label{eq:maxfunctineq}\|\mathcal{M}_{\vec{\alpha}}f\|_{p,\vec{\alpha}}\le C_{p,\vec{\alpha}}\|f\|_{p,\vec{\alpha}}.\Ee 
Using the mean value formula in each variable, we obtain that there is a constant $C>0$ such that for any $z\in \Pi^+	\times  \Pi^+$,
$$|f(z)|\le \frac{C}{\mathcal V_{\vec{\alpha}}(Q_R)}\int_{Q_R}|f(w)|d\mathcal V_{\vec{\alpha}}(w)$$
provided $f$ is analytic on $\Pi^+	\times  \Pi^+$; where for $z=(z_1,z_2)$, $Q_R$ is such that $z_j$ is the centre of $Q_{I_j}$, $R=I_1\times I_j$.
\vskip .2cm
It follows in particular that there is a constant $C>0$ such that for any $f$ analytic on $\Pi^+	\times  \Pi^+$,
$$|f(z)|\le C\mathcal{M}_{\vec{\alpha}}f(z),\,\,\,\textrm{for any}\,\,\,z\in \Pi^+	\times  \Pi^+.$$
Let us prove the following.
\begin{proposition}\label{prop:strongmaxembed}
Let $\vec{\alpha}=(\alpha_1,\alpha_2)$, with $\alpha_1,\alpha_2>-1$, and let $1<p\leq q<\infty$. Assume $\mu$ is $(\frac qp,\vec{\alpha})$-Carleson measure. Then there is a constant $C>0$ such that for any $f\in L^p(\Pi^+	\times  \Pi^+, d\mathcal V_{\vec{\alpha}}(w))$,
$$\int_{\Pi^+	\times  \Pi^+}\left(\mathcal{M}_{\vec{\alpha}}f(z)\right)^qd\mu(z)\le C\left(\int_{\Pi^+	\times  \Pi^+}|f(z)|^pd\mathcal V_{\vec{\alpha}}(z)\right)^{\frac qp}
.$$
\end{proposition}
\begin{proof}
Consider the following dyadic grids
$$\mathcal{D}^{\beta}:=\{2^j([0,1)+m+(-1)^j\beta):m,j\in \mathbb{Z}\}\,\,\,\textrm{for}\,\,\,\beta\in \left \{0,\frac 13\right \}.$$
For $\beta=0$, $\mathcal{D}^{\beta}=\mathcal{D}^0$ is the standard dyadic grid of $\mathbb{R}$ denoted $\mathcal{D}$.
\vskip .2cm
For $\vec{\beta}=(\beta_1,\beta_2)$, $\beta_j\in \left \{0,\frac 13\right \}$, we define the dyadic strong maximal function $\mathcal{M}_{\vec{\alpha}}^{d,\vec{\beta}}$ by
$$\mathcal{M}_{\vec{\alpha}}^{d,\vec{\beta}}f(z):=\sup_{R\in \mathcal{D}^{\vec{\beta}}}\frac{\chi_{Q_R}(z)}{\mathcal V_{\vec{\alpha}}(Q_R)}\int_{Q_R}|f(w)|d\mathcal V_{\vec{\alpha}}(w),$$
$\mathcal{D}^{\vec{\beta}}=\mathcal{D}^{\beta_1}\times \mathcal{D}^{\beta_2}$.
\vskip .2cm
We recall for any interval $I$ of $\mathbb{R}$, there exists an interval $J\in \mathcal{D}^{\beta}$ for some $\beta\in \{0,\frac 13\}$, such that $I\subseteq J$ and $|J|\le 6|I|$ (see \cite{PR}). It follows that the proposition will follow if we can prove that for $\mu$ an $(\frac qp, \vec{\alpha})$-Carleson measure, we can find a constant $C>0$ such that 
$$\int_{\Pi^+	\times  \Pi^+}\left(\mathcal{M}_{\vec{\alpha}}^{d,\vec{\beta}}f(z)\right)^qd\mu(z)\le C\left(\int_{\Pi^+	\times  \Pi^+}|f(z)|^pd\mathcal V_{\vec{\alpha}}(z)\right)^{\frac qp}
$$
for any $\vec{\beta}\in \{0,\frac 13\}^2$.
\vskip .2cm
For simplicity of presentation and because our proof works the same for each product dyadic grid, we restrict to the case $\vec{\beta}=(0,0)$ and denote by $\mathcal{M}_{\vec{\alpha}}^{d}$ the corresponding dyadic strong maximal function. 
For $R=I\times J\in\mathcal{D}\times \mathcal{D}=\mathcal{D}^{(0,0)}$, by top half of the Carleson rectangle $Q_R$, we will mean the set $$T_R:=T_I\times T_J.$$
We observe that unlike in the one parameter case, for $R_1\in \mathcal{D}\times \mathcal{D}$ and $R_2\in \mathcal{D}\times \mathcal{D}$, we have $$R_1\cap R_2\in \{\emptyset, R_1,R_2, R'\}$$
with $R'\in \mathcal{D}\times \mathcal{D}\,\,\,\textrm{such that} \,\,\,R'\subsetneq R_1 \,\,\,\textrm{and} \,\,\,R'\subsetneq R_2.$ 
Nevertheless, we still have that the family $\{T_R\}_{R\in \mathcal{D}\times \mathcal{D}}$ forms a tiling of $\Pi^+\times\Pi^+$. Indeed, one observes that even if two rectangles intersect with their intersection strictly contained in each of them, the corresponding top halves of the Carleson rectangles are still disjoint.

Let $f\in L^p(\Pi^+	\times  \Pi^+,d\mathcal V_{\vec{\alpha}}(w))$. For each integer $k$, define
$$E_k:=\{z\in \Pi^+	\times  \Pi^+: 2^k<\mathcal{M}_{\vec{\alpha}}^{d}f(z)\le 2^{k+1}\}.$$ 
Denote by $\mathcal{F}_k$ the family of all rectangles $R\in \mathcal{D}\times \mathcal{D}$ such that
$$\frac{1}{\mathcal V_{\vec{\alpha}}(Q_R)}\int_{Q_R}|f(w)|d\mathcal V_{\vec{\alpha}}(w)>2^k.$$
Then clearly, we have that 
\Be\label{eq:rectanglemax}
E_k\subseteq\bigcup_{R\in \mathcal{F}_k\setminus\mathcal{F}_{k+1}}Q_R.
\Ee
Indeed, let $z\in E_k$ and suppose that there is no dyadic rectangle $R$ with $z\in Q_R$ such that $$2^k<\frac{1}{\mathcal V_{\vec{\alpha}}(Q_R)}\int_{Q_R}|f(w)|d\mathcal V_{\vec{\alpha}}(w)\le 2^{k+1}.$$
Then for any $R\in \mathcal{D}\times \mathcal{D}$ such that $z\in Q_R$, either $$\frac{1}{\mathcal V_{\vec{\alpha}}(Q_R)}\int_{Q_R}|f(w)|d\mathcal V_{\vec{\alpha}}(w)\le 2^k$$
or $$\frac{1}{\mathcal V_{\vec{\alpha}}(Q_R)}\int_{Q_R}|f(w)|d\mathcal V_{\vec{\alpha}}(w)>2^{k+1}.$$
Hence either $\mathcal{M}_{\vec{\alpha}}^{d}f(z)\le 2^k$ or $\mathcal{M}_{\vec{\alpha}}^{d}f(z)>2^{k+1}$. This contradicts the fact that $z\in E_k$.
\vskip 2truemm
It first follows that
$$
\begin{array}{clcr}
L &:=& \int_{\Pi^+	\times  \Pi^+}\left(\mathcal{M}_{\vec{\alpha}}^{d}f(z)\right)^qd\mu(z)\\ &=& \sum_k\int_{E_k}\left(\mathcal{M}_{\vec{\alpha}}^{d}f(z)\right)^qd\mu(z)\\ &\le& 2^q\sum_k 2^{kq}\mu(E_k)\\ &\le& 2^q\sum_k \sum_{R\in \mathcal{F}_k\setminus\mathcal{F}_{k+1}}2^{kq}\mu(Q_R)\\ &\le& C2^q\sum_k\sum_{R\in \mathcal{F}_k\setminus\mathcal{F}_{k+1}}2^{kq}\left(\mathcal V_{\vec{\alpha}}(Q_R)\right)^{q/p}.
\end {array}
$$
We next use the property (\ref{eq:rectanglemax}) of rectangles in $\mathcal{F}_k$ and the equivalence $\mathcal V_{\vec{\alpha}}(Q_R)\approx \mathcal V_{\vec{\alpha}}(T_R)$ to obtain
$$
\begin{array}{clcr}
L &\le& C2^q\sum_k\sum_{R\in \mathcal{F}_k\setminus\mathcal{F}_{k+1}}\left(\frac{1}{\mathcal V_{\vec{\alpha}}(Q_R)}\int_{Q_R}|f(w)|d\mathcal V_{\vec{\alpha}}(w)\right)^q\left(\mathcal V_{\vec{\alpha}}(Q_R)\right)^{q/p}\\ &\le& C2^q\left(\sum_k\sum_{R\in \mathcal{F}_k\setminus\mathcal{F}_{k+1}}\left(\frac{1}{\mathcal V_{\vec{\alpha}}(Q_R)}\int_{Q_R}|f(w)|d\mathcal V_{\vec{\alpha}}(w)\right)^p\mathcal V_{\vec{\alpha}}(Q_R)\right)^{q/p}\\ &\approx& C2^q\left(\sum_k\sum_{R\in \mathcal{F}_k\setminus\mathcal{F}_{k+1}}\left(\frac{1}{\mathcal V_{\vec{\alpha}}(Q_R)}\int_{Q_R}|f(w)|d\mathcal V_{\vec{\alpha}}(w)\right)^p\mathcal V_{\vec{\alpha}}(T_R)\right)^{q/p}.
\end {array}
$$
Recalling that $\{T_R\}_{R\in \mathcal{D}\times \mathcal{D}}$ forms a tiling of $\Pi^+\times \Pi^+$ and using (\ref{eq:maxfunctineq}), we finally obtain
\Beas
L &:=& \int_{\Pi^+	\times  \Pi^+}\left(\mathcal{M}_{\vec{\alpha}}^{d}f(z)\right)^qd\mu(z)\\
&\le& C2^q\left(\sum_k\sum_{R\in \mathcal{F}_k\setminus\mathcal{F}_{k+1}}\int_{T_R}\left(\frac{1}{\mathcal V_{\vec{\alpha}}(Q_R)}\int_{Q_R}|f(w)|d\mathcal V_{\vec{\alpha}}(w)\right)^p d\mathcal V_{\vec{\alpha}}(z)\right)^{q/p}\\
&\le& C2^q \left (\sum_{R\in \mathcal D \times \mathcal D} \int_{T_R} (\mathcal{M}_{\vec{\alpha}}f(z))^pd\mathcal V_{\vec{\alpha}}(z)\right)^{q/p}\\
&=& C2^q\left(\int_{\Pi^+\times \Pi^+}(\mathcal{M}_{\vec{\alpha}}f(z))^pd\mathcal V_{\vec{\alpha}}(z)\right)^{q/p}\\ &\le& C\left(\int_{\Pi^+\times \Pi^+}|f(z)|^pd\mathcal V_{\vec{\alpha}}(z)\right)^{q/p}.
\Eeas
The proof of the proposition is complete.
\end{proof}
Finally, we prove Theorem \ref{thm:carlbergprod}.
\begin{proof}[Proof of Theorem \ref{thm:carlbergprod}]
That (a)$\Rightarrow(b)$ follows from the observations made above and Proposition \ref{prop:strongmaxembed}.
\vskip .2cm
(b)$\Rightarrow$(a): Assume that there is a constant $C>0$ such that for every $f\in A_{\vec{\alpha}}^p(\Pi^+	\times  \Pi^+)$,
\Be\label{eq:ineqtest}\int_{\Pi^+	\times  \Pi^+}|f(z)|^qd\mu(z)\le C\left(\int_{\Pi^+	\times  \Pi^+}|f(z)|^pd\mathcal V_{\vec{\alpha}}(z)\right)^{\frac qp}
\Ee
Let $Q_{I_1\times I_2}$ be a fixed Carleson box, and let $w=(w_1,w_2)$, where $w_j$ is the centre of the Carleson square $Q_{I_j}$, $j=1,2$. We consider the function $f_w$ defined on $\Pi^+	\times  \Pi^+$ by
$$f_w(z):=\left(\frac{(\Im m \hskip 1truemm w_1)^{1+\frac {\alpha_1}2}}{(z_1-\overline{w}_1)^{2+\alpha_1}}\frac{(\Im m \hskip 1truemm w_2)^{1+\frac {\alpha_2}2}}{(z_2-\overline{w}_2)^{2+\alpha_2}}\right)^{\frac 2p},\hskip 2truemm z=(z_1,z_2)\in \Pi^+	\times  \Pi^+.$$
Then $f_w$ is uniformly in $A_{\vec{\alpha}}^p(\Pi^+	\times  \Pi^+)$. We observe for any $z_j\in Q_{I_j}$, $|z_j-\overline{w}_j|\approx \Im m \hskip 1truemm w_j$. Testing the inequality (\ref{eq:ineqtest}) with our choice $f_w$, we obtain
\Beas  \frac{\mu(Q_{I_1\times I_2})}{(\Im m\hskip 1truemm w_1)^{(2+\alpha_1)\frac{q}{p}}(\Im m\hskip 1truemm w_2)^{(2+\alpha_2)\frac{q}{p}}} &\lesssim& \int_{\Pi^+	\times  \Pi^+}|f_w(z)|^qd\mu(z)\\ &\le& C\left(\int_{\Pi^+	\times  \Pi^+}|f_w(z)|^pd\mathcal V_{\vec{\alpha}}(z)\right)^{\frac qp}\\ &\le& C.
\Eeas
The proof is complete.
\end{proof}
\bibliographystyle{plain}

\end{document}